\documentclass[reqno]{amsart}

\usepackage[applemac]{inputenc}

\usepackage{amssymb}
\usepackage{graphicx}
\usepackage[cmtip,all]{xy}
\usepackage{verbatim}
\usepackage{tikz-cd}
\usepackage{mathtools}
\usepackage{graphicx}
\usepackage{mathrsfs}
\usepackage{geometry}

\DeclareMathAlphabet{\mathpzc}{OT1}{pzc}{m}{it}

\newtheorem{theorem}{Theorem}[section]
\newtheorem{lemma}[theorem]{Lemma}

\theoremstyle{definition}
\newtheorem{definition}[theorem]{Definition}
\newtheorem{example}[theorem]{Example}

\theoremstyle{remark}
\newtheorem{remark}[theorem]{Remark}

\numberwithin{equation}{section}

\makeatletter
\DeclareRobustCommand{\cev}[1]{%
  \mathpalette\do@cev{#1}%
}
\newcommand{\do@cev}[2]{%
  \fix@cev{#1}{+}%
  \reflectbox{$\m@th#1\vec{\reflectbox{$\fix@cev{#1}{-}\m@th#1#2\fix@cev{#1}{+}$}}$}%
  \fix@cev{#1}{-}%
}
\newcommand{\fix@cev}[2]{%
  \ifx#1\displaystyle
    \mkern#23mu
  \else
    \ifx#1\textstyle
      \mkern#23mu
    \else
      \ifx#1\scriptstyle
        \mkern#22mu
      \else
        \mkern#22mu
      \fi
    \fi
  \fi
}

\makeatother

 \newcommand{\virgolette}{``}

\newcommand{\slantone}[2]{{\raisebox{.1em}{$#1$}\left/\raisebox{-.1em}{$#2$}\right.}}
\newcommand*{\defeq}{\mathrel{\vcenter{\baselineskip0.5ex \lineskiplimit0pt
                     \hbox{\scriptsize.}\hbox{\scriptsize.}}}%
                     =}

\newcommand\asim{\mathrel{%
  \ooalign{\raise0.1ex\hbox{$\sim$}\cr\hidewidth\raise-0.8ex\hbox{\scalebox{0.9}{$\scriptstyle{x}$}}\hidewidth\cr}}}

\newcommand{\mani}{\ensuremath{\mathpzc{M}}}
\newcommand{\manir}{\ensuremath{\mathpzc{M}_{red}}}
\newcommand{\stsheaf}{\ensuremath{\mathcal{O}_{\mathpzc{M}}}}

\newcommand{\beq}{\begin{equation}}
\newcommand{\eeq}{\end{equation}}
\newcommand{\bear}{\begin{eqnarray}}
\newcommand{\eear}{\end{eqnarray}}

\begin{document}

\title{A Note on Super Koszul Complex and the Berezinian}


\author{Simone Noja}
\address{Universit\"{a}t Heidelberg}
\curraddr{Im Neuenheimer Feld 205, 69120 Heidelberg, Germany}
\email{noja@mathi.uni-heidelberg.de}

\author{Riccardo Re}
\address{Università degli Studi dell'Insubria}
\curraddr{Via Valleggio 11, 22100 Como, Italy}
\email{riccardo.re@uninsubria.it}



\begin{abstract} We construct the super Koszul complex of a free supercommutative $A$-module $V$ of rank $p|q$ and prove that its homology is concentrated in a single degree and it yields an exact resolution of $A$. We then study the dual of the super Koszul complex and show that its homology is concentrated in a single degree as well and isomorphic to $\Pi^{p+q} A$, with $\Pi$ the parity changing functor. Finally, we show that, given an automorphism of $V$, the induced transformation on the only non-trivial homology class of the dual of the super Koszul complex is given by the multiplication by the Berezinian of the automorphism, thus relating this homology group with the Berezinian module of $V$.

\end{abstract}

\maketitle

\tableofcontents

\section{Introduction}

\noindent The definition of the \emph{Koszul Complex}, whose first introduction as an example of a complex of free modules over a commutative ring $A$ dates back to Hilbert, marks the advent of homological methods in commutative algebra in the early `50s. Since it is a resolution of $k=A/m$ with $(A,m)$ a regular local ring, or of $R=S/(x_0,\ldots,x_r)$, with $S=R[x_0,\ldots,x_r]$, it allowed the computation of derived functors like $ Tor^i(A,M)$ and $Ext^i(A,M)$ for any $S$-module $M$,  hence of important homological invariants - like the \emph{projective dimension} of a module (or its Koszul homology) - and their relations to more classical concepts, like for example the \emph{depth of a module}.\\ 
The Koszul complex may also be constructed in a non-commutative setting, \emph{i.e.}\ for any left $A$-module $M$ with $A$ any ring and elements $x_1,\ldots,x_r$ in $A$, there exists a Koszul complex $K(x_1,\ldots, x_r,M)$, when $x_i$ are pairwise commuting when viewed as multiplication maps $x_i:M\to M$. In this paper we consider the case of a commutative \emph{superalgebra} $S$ and therefore we take the $x_i$'s to be supercommutative. 
We first provide the construction of the Koszul complex $K(x_1,\ldots,x_r)$ in the supercommutative case and then we compute its homology in the universal case $S=A[x_1,\ldots,x_r]$, with $A$ supercommutative. In doing so, we revisit the proof that it is a resolution of $A$ as a $S$-module, and then we study the \emph{dual} complex $K^\ast(x_1,\ldots,x_r)$ and calculate its homology, hence computing $Ext^i(A,S)$. In particular, this produces the \emph{Berezinian} of the free $A$-module $F=\bigoplus_i Ax_i$. \\
Although it is fair to say that results regarding the Koszul complex in a \virgolette super setting'' have previously appeared - see \cite{Manin} -, we are not aware of a complete and detailed treatment of this fundamental construction in the existing superalgebra or supergeometry literature. We take the chance to fill such a gap with the present paper, which provides also a completely self-contained exposition of the subject. 
Further, we remark that the supercommutative setting for the Koszul complex includes of course the commutative case, in a way that clearly shows the intimate supercommutative nature of the \emph{classical} Koszul complex. Moreover, the magical \emph{self-duality} of the classical Koszul complex - allowing for example for the rich theory of \emph{complete intersections} or, more generally, \emph{Gorenstein rings} in commutative algebra -, is put in the right setting within the treatment of the \emph{super} Koszul complex of this paper, where it is made clear that the dual of a super Koszul complex is not in general isomorphic to \emph{the same} Koszul complex, but to another one, depending on the numbers of even and odd  variables involved. Finally, we stress that the given construction via homology of the super Koszul complex, provides a completely invariant construction of the Berezinian of a free-module, a crucial building block for modern supersymmetric theories in theoretical physics. \\

\noindent The paper is structured as follows. In section two we establish our conventions and we provide the reader with some definitions and preliminary result. In particular, in Theorem \ref{kosclas}, we use our superalgebraic setting to compute the homology of the classic Koszul complex, which will be used later on in the paper. In section three, we construct the super Koszul complex of a free supercommutative $A$-module $V$ and we compute its homology in Theorem \ref{homologykos}. The result is built upon Lemma \ref{lemma1} and Lemma \ref{lemma2}, which compute the homotopy operator of the differential of the super Koszul complex. In this respect, in Remark \ref{remark} at the end of section three, we address the differences with the classical Koszul homology and we briefly discuss the interesting case of characteristic $p$ in the superalgebraic setting by means on an example. In section four we introduce the \emph{dual} of the super Koszul complex and we briefly discuss the functoriality of the construction. Then we proceed to compute the homology of the dual of the super Koszul complex in Theorem \ref{dualhomology}, whose proof is based on the ancillary result Theorem \ref{kosclas}. Finally, in the last section we make contact with the Berezinian module of $V$: in particular, we prove that given an automorphism of $V$, a representative of the homology of the dual of the super Koszul complex transforms with the inverse of the Berezinian of the automorphism. This allows to identify the (dual of) Berezinian module of $V$ with the only non-trivial homology module of the dual of the super Koszul complex of $V$.\\

\noindent \emph{Addendum}: soon after this paper appeared as a preprint, Prof.\ Ogievetsky makes us aware that the super Koszul complex and its dual first appeared in \cite{OP}, by him and I.B.\ Penkov - see Corollary 4 therein -, but a detailed treatment has never indeed appeared in the literature. We would like to thank him for pointing out this reference to us.

\section{Preliminary Definitions and the Koszul Complex via Superalgebra}

\noindent In this section we recall some elements of superalgebra that will be used in the following. For a thorough exposition of superalgebra we refer to the classical \cite{Manin} or the recent \cite{CCF}. \\
\noindent Let $A$ be any superalgebra of characteristic $0$ and let $V = A^{p|q}$ be a free $A$-supermodule with basis given by $\{x_1,\ldots, x_p | \theta_1, \ldots, \theta_q \}$, where the $\mathbb{Z}_2$-grading, or \emph{parity}, reads $|x_i| = 0 $ and $|\theta_j| = 1$ for any $i= 1, \ldots, p$ and $j = 1, \ldots, q.$ We define 
\bear \label{erre}
R \defeq \bigoplus_{k\geq 0} R_k \quad \mbox{with} \quad R_k \defeq Sym^k V,
\eear
where $Sym^k(\,\cdot\, ) : \mathbf{SMod}_A  \rightarrow \mathbf{SMod}_A$ is the (super)symmetric $k$-power functor from the category of $A$-supermodules to itself. Henceforth we will refer to $Sym^k (\, \cdot \,)$ simply as the $k$-\emph{symmetric} product. We observe that $R$ has a structure of a $\mathbb{Z}$-graded $A$-algebra, where the products $R_i \otimes R_j \rightarrow R_{i+j} $ are induced by the symmetric product, and also of a $\mathbb{Z}_2$-graded commutative (or supercommutative) $A$-algebra - we say that it is an $A$-superalgebra - where the grading is induced by that of $V$.\\
It is worth noticing that $A$ is a $R$-module thanks to the short exact sequence 
\bear
\xymatrix{
0 \ar[r] & I_1 \ar[r] & R \ar[r] & A \ar[r] & 0, 
}
\eear
where $I_1 \defeq \bigoplus_{k\geq 1} Sym^k V$ is the (maximal) ideal of $R$ generated by $V \cong Sym^1 V \subset R$ and, hence, $A \cong \slantone{R}{I_1R}.$ The ideal $I_1$ has the following presentation:
\bear \label{pres}
\xymatrix{
R \otimes \Pi V \ar[r]^{\quad p} &  I_1 \ar[r] &0,
}
\eear 
where $\Pi : \mathbf{SMod}_A  \rightarrow \mathbf{SMod}_A $ is the \emph{parity changing functor}, acting on objects by simply reversing their parity. In the \eqref{pres} the (surjective) morphism $p : R \otimes \Pi V \rightarrow I_1$ is defined as follows on even and odd generators
\begin{align} \label{delta1}
\xymatrix@R=1.5pt{
1 \otimes \pi x_i \ar@{|->}[r]^{\quad p} & x_i,  \\
1 \otimes \pi \theta_i \ar@{|->}[r]^{\quad p} &  \theta_j.
}
\end{align}
Notice that $p$ is an \emph{odd} morphism, as it reverses parity.\\
Given $V$ as above we define its \emph{dual} as $
V^\ast \defeq {Hom}_{A} (V, A).$ This defines again a supercommutative $A$-module with parity splitting given by $V^\ast = Hom_{A} (V, A)_0 \oplus Hom_{A} (V, A)_1$, \emph{i.e.}\ the even and odd $A$-linear maps. For any $i\leq k$ there exists a pairing given by 
\bear
\xymatrix@R=1.5pt{
\langle \cdot , \cdot \rangle : Sym^i V^\ast \otimes Sym^k V \ar[r] & Sym^{k-i}V \\
t^\ast_i \otimes s_k \ar@{|->}[r] & \langle t^\ast_i, s_k \rangle. 
}
\eear
If $V$ has basis given by $\{x_1, \ldots, x_p | \theta_1, \ldots, \theta_q \}$ as above, this can be obtained by considering the dual space $V^\ast$ as the space having basis given by $\{\partial_{x_1}, \ldots, \partial_{x_p} | \partial_{\theta_1}, \ldots, \partial_{\theta_q} \}$, so that one has the identification 
\bear
Sym^i V^\ast \defeq \{ D : Sym^{k\geq i} V \rightarrow Sym^{k-i} V: D\mbox{ is a homogeneous operator of order } i \}.
\eear 
Notice in particular that, for $k=i$ one has the \emph{duality pairing} $Sym^k V^\ast \otimes Sym^k V \rightarrow {A}$.\\

\noindent Notably, the above superalgebraic setting can be used to reinterpret the construction of the \virgolette ordinary'' Koszul complex (see for example \cite{Eisenbud} for an extended treatment of the subject), and compute its homology in a very economic and elegant way. \\
Let us consider indeed a free-module $A^N$ for a certain ring or algebra $A$ with a basis given by $\{ x_1, \ldots, x_N \}$. Then one can can construct the free supermodule $A^{N|N} = A^N \oplus \Pi A^N$ with a basis given by $\{x_1, \ldots, x_N | y_1, \ldots, y_N \}$, where $|x_i| = 0$, \emph{i.e.}\ the $x_i$'s are even, and $|y_i| = 1$, \emph{i.e.}\ the $y_i$'s are odd, for any $i= 1, \ldots, N$: in this sense, the $y_i$'s can be defined as $y_i \defeq \pi x_i$, just by changing the parity of the generator of $A^N$. Notice that the $x_i$'s and the $y_i$'s are not $A$-linearly dependent.\\
Here, the supermodule $A^{N|N}$ plays the role of $V$ introduced in the above construction, so that $R$ can be written as 
\bear \label{Rkos}
R = B \oplus (B \cdot U)  \oplus ( B \cdot Sym^2 U ) \oplus \ldots \oplus ( B \cdot Sym^N U ) = B \otimes_A \bigoplus_{k = 0}^{N} Sym^k U.
\eear
where we have defined
\bear
B \defeq A [x_1, \ldots, x_N] \quad \mbox{and} \quad U \defeq A [y_1, \ldots, y_N]. 
\eear
Notice that, classically, $B$ and $U$ can be seen respectively as the \emph{symmetric} and \emph{exterior} algebra over a set of $N$ generators (over $A$). Let us define the following multiplication operator on $R$:
\bear
\xymatrix@R=1.5pt{
d : R \ar[r] & R \\
b \cdot F \ar@{|->}[r] & d (b \cdot F) \defeq (\sum_{i=1}^N x_i \cdot y_i) (b \cdot F),
}
\eear
where $b\in B$ and $F \in Sym^\bullet U$. In other words the action of $d$ corresponds to the multiplication by the element $\sum_i x_i \cdot y_i$ in $R$: it is immediate to observe the operator is indeed nilpotent, \emph{i.e.}\ $d \circ d = 0$ because the multiplying element $\sum_i x_i \cdot y_i$ is odd. It follows that $d$ makes $R$ into an actual complex $R \defeq K_\bullet,$ where the $\mathbb{Z}$-grading is induced by the (super)symmetric powers of $U$ as in \eqref{Rkos} and the pair $(K_\bullet , d)$ is a \emph{differentially graded} (dg) $B$-algebra, which we call the (dual of the) \emph{Koszul complex} associated to $A^N$, or Koszul complex of $A^N$ for short. The homology of the Koszul complex is concentrated in degree $N$ as the following theorem shows. 
\begin{theorem} \label{kosclas}Let $(K_\bullet, d)$ the Koszul complex associate to $A^N$ for some $A$. Then the homology of $(R_\bullet, d)$ is concentrated in degree $N$. More in particular, we have 
\bear
H_i ((K_\bullet, d)) \cong \left \{ \begin{array}{ccc}
A \cdot y_1 \ldots y_N & & i = N \\
0 & & \mbox{else}
\end{array}
\right. 
\eear
\end{theorem}
\begin{proof} We construct a homotopy for $d : K_i \rightarrow K_{i+1}.$ To this end let us consider $h^K \defeq \sum_{i=1}^N \partial_{x_i} \partial_{y_i} : K_{i} \rightarrow K_{i-1}.$ Without loss of generality we can restrict to homogeneous elements $b\in B$ and $F \in Sym^\bullet U$ and we compute:
\begin{align}
h^K  d (b \cdot F) & = \sum_{j = 1}^N \partial_{x_j } \cdot \partial_{y_j} \left (\sum_{i = 1}^N x_i \cdot y_i ( b \cdot F ) \right ) = \sum_{i, j = 1}^N \partial_{x_j} \left ( x_i b \right) \partial_{y_j} \left ( y_i F \right ) \nonumber \\
& = \sum_{i,j = 1}^N \left ( \delta_{ij} b + x_i \cdot \partial_{x_j } f \right ) \left (\delta_{ij} F - y_i \cdot \partial_{y_j} F \right )  \nonumber \\
& = \sum_{i} \left ( \delta_{ii} b \cdot F - b \cdot  (y_i \partial_{y_i}  F ) - (x_i \partial_{x_i} b) \cdot F - (x_i \cdot y_i ) (\partial_{x_i} \partial_{y_i}) b \cdot F \right ) \nonumber \\
& = N b\cdot F - \deg (F) \, b\cdot F + \deg (b) \, b \cdot F - d \,h^K (b \cdot F).
\end{align}
This leads to 
\bear
h^K d + d \, h^K = N - \deg (F) + \deg (b).
\eear
Since $\deg (F) \leq N$ we have that $N - \deg (F) \geq 0$, so that if $\deg (b) > 0$ the sum above is never zero, and for any such pair $(\deg (b), \deg (F))$ with $\deg (b) >0$ and $0 \leq \deg (F) \leq N$ we can define a homotopy operator for $d$ as
\bear
h^K_{\deg (b), \deg (F)} \defeq \frac{h^K}{N - \deg (F) + \deg (b)} : K_{i} \rightarrow K_{i-1}. 
\eear 
The only instance in which the homotopy fails is when $\deg (F) = N$ and $\deg (b) = 0$: the generator of the corresponding module is $A\cdot y_1 \ldots y_n$ which is clearly in the kernel of $d$.
\end{proof}
\noindent The geometrical upshot of this important theorem is that the \emph{determinant} or \emph{canonical} module $\wedge^N A^N$ related to $A^N$ emerges as the (co)homology of the Koszul complex of $A^N$.  We will make use of this result later on in the paper.

\section{The Super Koszul Complex and its Homology}

\noindent In this section we define a super analog of the ordinary, \emph{i.e.}\ commutative, \emph{Koszul complex}, and we shall see that, as in the ordinary commutative setting, also in the supercommutative setting the Koszul complex yields a resolution of $A$ as a $R$-module, with $R$ given by \eqref{erre}. \\
For future convenience, given basis of $\{x_1, \ldots, x_p | \theta_1 , \ldots, \theta_q\}$ of $V = A^{p|q}$, we introduce the basis $\{ \ell_1, \ldots, \ell_q | \chi_i, \ldots, \chi_p \}$ of $\Pi V$, where the have set $\ell_j \defeq \pi \theta_j $ and $\chi_i \defeq \pi x_i$, so that $|\ell_j | = 0$ and $|\chi_i | = 1 $ for any $i=1, \ldots, p$ and $j=1, \ldots, q$. Notice that if $V$ has dimension $p|q$, then $\Pi V$ has dimension $q|p.$\\
As a warm-up, with reference to the previous section, let us consider the following composition of maps:
\bear
\xymatrix{
R \otimes Sym^2 \Pi V \ar[r]^{\delta_2} & R \otimes Sym^1 \Pi V \ar[r]^{\qquad \; \; \delta_1} & R, 
}
\eear
where $\delta_1 : R \otimes \Pi V \rightarrow R$ is given by the composition of the presentation $p : R \otimes \Pi V \rightarrow I_1 $ of the ideal $I_1 \subset R$ with the immersion $i : I_1 \hookrightarrow R$, so that $\delta_1 \defeq i \circ p : R \otimes \Pi V \rightarrow R$ acts as already defined in equation \eqref{delta1}, which employing the $\{\ell_j |\chi_i \}$-notation introduced above now reads:
\bear
\xymatrix@R=1.5pt{
1 \otimes \ell_j \ar@{|->}[r]^{\; \; \delta_1} &  \theta_j, \\
1 \otimes \chi_i \ar@{|->}[r]^{\; \; \delta_1} & x_i.
}
\eear
The map $\delta_2 : R\otimes Sym^2 \Pi V \rightarrow R \otimes \Pi V$ is defined on a basis of $R \otimes Sym^2 \Pi V$ as follows: 
\begin{align}
\xymatrix@R=1.5pt{
1 \otimes (\ell_i \odot \ell_j) = 1 \otimes (\pi \theta_i \odot \pi \theta_j) \ar@{|->}[r]^{\delta_2 \; \quad }  & \theta_i \otimes \pi \theta_j + \theta_j \otimes \pi \theta_i = \theta_i \otimes \ell_j + \theta_j \otimes \ell_i, \\
1 \otimes (\ell_i \odot \chi_k)= 1 \otimes (\pi \theta_i \odot \pi x_k)  \ar@{|->}[r]^{\delta_2 \; \quad } & \theta_i \otimes \pi x_k + x_k \otimes \pi \theta_i = \theta_i \otimes \chi_k + x_k \otimes \ell_j \\
1 \otimes (\chi_k \odot \chi_l) = 1 \otimes (\pi x_k \odot \pi x_l ) \ar@{|->}[r]^{\delta_2 \; \quad } & x_k \otimes \pi x_l - x_l \otimes \pi x_k = x_k \otimes \chi_l - x_l \otimes \chi_k.
}
\end{align}
It is straightforward to observe that these elements are in kernel of the map $\delta_1$,  
so that one has that $\delta_2 \circ \delta_1 = 0.$ This is not by accident and indeed the above construction can be made general. \\
Having already defined $R \defeq \oplus_{k\geq 0} Sym^k V$, we further introduce 
\bear
R^\pi \defeq \bigoplus_{k\geq 0} R_k^\pi \quad \mbox{with} \quad R_k^\pi \defeq Sym^k \Pi V,
\eear
and in turn, we define the tensor product of $R$ and $R^\pi$ over $A$: 
\bear
\mathcal{K}_\bullet \defeq \bigoplus_{k \geq 0 } \mathcal{K}_{-k} = R \otimes_A \bigoplus_{k \geq 0 } R^{\pi}_k  = R \otimes_A R^\pi.
\eear
Clearly, $\mathcal{K}_\bullet$ is an $A$-superalgebra, as both $R$ and $R^\pi$ are. We now introduce the following $A$-superalgebra homomorphism:
\bear
\xymatrix@R=1.5pt{
\delta : \mathcal{K}_\bullet = R \otimes_A R^\pi \ar[r] & \mathcal{K}_\bullet = R \otimes_A R^\pi \\
r \otimes r^\pi \ar@{|->}[r] & \delta (r \otimes r^\pi ) \defeq \left ( \sum_{i=1}^p x_i \otimes \partial_{\chi_i} + \sum_{j=1}^q \theta_j \otimes \partial_{\ell_j} \right ) (r \otimes r^\pi), 
}
\eear
where $r \in R $ and $r^\pi \in R^\pi.$ It is immediate to observe the following facts: 
\begin{enumerate}
\item with respect to the supercommutative structure of $\mathcal{K}_\bullet$, that is with respect to the $\mathbb{Z}_2$-gradation, $\delta$ is \emph{odd}, \emph{i.e.} $|\delta|= 1$;
\item with respect to the $\mathbb{Z}$-gradation of $\mathcal{K}_\bullet$ as an $R$-module, $\delta$ is homogeneous of degree $-1$.
\item $\delta$ acts as a \emph{derivation} only on the factor $R^\pi$. It follows that $\delta$ is $R$-linear on $\mathcal{K}_\bullet = R \otimes_A R^\pi$ endowed with the structure of an $R$-module (actually, $R$-algebra) induces by its factor $R$.
\end{enumerate}
We have the following theorem.
\begin{theorem} The pair $(\mathcal{K}_\bullet, \delta)$ defines a differentially graded \mbox{\emph{(dg)}} $R$-algebra.
\end{theorem}
\begin{proof} We have already observed that $\delta$ is $R$-linear. We need to prove that $\delta^2 \defeq \delta \circ \delta = 0$. We observe that $\mathcal{K}_\bullet $ is generated as an $R$-algebra by $Sym^1\Pi V \cong \Pi V$, which is $\mathcal{K}_{-1}$ upon tensoring with $R$, \emph{i.e.} looking at $\mathcal{K}_\bullet$ as an $R$-algebra. It follows that it is enough to verify that $\delta^2 (1 \otimes \ell_j) = 0$ and $\delta^2 (1 \otimes \chi_i)$ for any $j = 1, \ldots, q$ and $i = 1, \ldots, p$, where $\{ 1 \otimes \ell_j | 1 \otimes \chi_i \}$ are the generators of $\Pi V$, and then work by induction on the $\mathbb{Z}$-degree of $\mathcal{K}_\bullet.$  \\
Obviously, one has
\begin{align}
\delta^2 (1 \otimes \ell_j) = \delta (x_j \otimes 1) = 0, \nonumber \\
\delta^2 (1 \otimes \chi_j) = \delta (\theta_j \otimes 1) = 0, 
\end{align}
which settle the case $\mathcal{K}_{-1}.$ Let us now assume that $s \in \mathcal{K}_{-1}$ and $t \in \mathcal{K}_{-k}$ for $k \geq 1$ satisfying $\delta^2 (t) = 0$ by induction hypothesis. By Leibniz rule, for an element $s\cdot t \in \mathcal{K}_{-k-1},$ one has
\begin{align}
\delta^{2} (s \cdot t) &= \delta^2 (s) t + (-1)^{|\delta||\delta (s)| } \delta (s) \delta (t) + (-1)^{|\delta| |s| } \delta (s) \delta (t) + (-1)^{2|\delta| |s|}s \delta^{2} (t)  \nonumber \\
& = (-1)^{|s| + 1 } \delta (s) \delta (t) + (-1)^{|s| } \delta (s) \delta (t) + s \delta^{2} (t) \nonumber \\
& = s \delta^2 (t) = 0
\end{align}
by induction hypothesis, recalling that $|\delta| = 1$ so that, in particular $|\delta (s)| = s +1.$  
\end{proof}
\noindent The previous theorem justifies the following definition.
\begin{definition}[Super Koszul Complex] Given any free $A$-module $V = A^{p|q}$ for any superalgebra $A$, we call the pair $(\mathcal{K}_\bullet, \delta)$ the \emph{super Koszul complex} associated to $V$: 
\bear
\xymatrix{
\cdots \ar[r]^{ \delta \qquad } & R \otimes Sym^k \Pi V \ar[r]^{\qquad \delta} & \cdots \ar[r]^{\delta \qquad } & R \otimes Sym^2 \Pi V \ar[r]^{\quad \delta} & R \otimes \Pi V \ar[r]^{\quad \delta} & R \ar[r] & 0.  
}
\eear
\end{definition}
\noindent We are interested into studying the homology of the super Koszul complex. In order to do this, we start with some preliminary considerations. \\
Let $f \in A[x_1, \ldots, x_p | \theta_1, \ldots, \theta_q] = Sym^\bullet V$ be a \emph{bi-homogeneous} polynomial, \emph{i.e.}\ homogeneous in the $x_i$'s and in the $\theta_j$'s. Then the map $f \mapsto \deg f$, which associates to $f$ its bi-homogeneous degree is well-defined. 
Now, let $E$ be the \emph{Euler vector fields,} \emph{i.e.} the differential operator defined 
\bear
E \defeq \sum_{i=1}^p x_i \partial_{x_i} + \sum_{j=1}^q \theta_q \partial_{\theta_q}.
\eear 
and acting on polynomials.
Then one has the following easy lemma, which mimic the ordinary one for the commutative case.
\begin{lemma} \label{Euler} Let $f(x_1,\ldots, x_p | \theta_1, \ldots, \theta_q) \in A[x_1, \ldots, x_p | \theta_1, \ldots, \theta_q] $ a bi-homogeneous polynomial and let $E \defeq \sum_{i=1}^p x_i \partial_{x_i} + \sum_{j=1}^q \theta_q \partial_{\theta_q}$ be the {Euler vector field}. Then
\bear
(\deg f ) f = E (f),
\eear
where $\deg f$ is the degree of the bi-homogeneous polynomial $f$. 
\end{lemma}
\begin{proof} Obvious, as if follows from the case of monomials.
\end{proof}
\noindent This ancillary result will be useful in the computation of the homology of the super Koszul complex. 
The crucial result in this direction is the construction of a \emph{homotopy operator} ${h}^{\mathcal{K}}_\bullet: \mathcal{K}_{\bullet} \rightarrow \mathcal{K}_{\bullet}$ for the differential of the super Koszul complex, with $h^{\mathcal{K}}_i : \mathcal{K}_{-i} \rightarrow \mathcal{K}_{-i-1}$ such that $h^{\mathcal{K}}_{i+1} \circ \delta_i + \delta_{i-1} \circ h^{\mathcal{K}}_i = id_{\mathcal{K}_i}$. We organize the construction of the homotopy in two consequential lemmas. 
\begin{lemma} \label{lemma1} Let $(\mathcal{K}_{\bullet}, \delta)$ be the super Koszul complex associated to some $V = A^{p|q}$ with basis as above. Then the operator $\varepsilon : \mathcal{K}_{\bullet} \rightarrow \mathcal{K}_{\bullet} $ defined as
\bear
\varepsilon \defeq \sum_{i=1}^p \partial_{x_i} \otimes \chi_i + \sum_{j=1}^q \partial_{\theta_j} \otimes \ell_j
\eear
is such that 
\bear
\delta \circ \varepsilon + \varepsilon \circ \delta = E \otimes id_{R^\pi} + id_R \otimes E^\pi,
\eear
for $E$ the Euler vector field acting on $R = Sym^\bullet V$ and $E^\pi$ the Euler vector field acting on $R^\pi = Sym^\ast \Pi V$.
\end{lemma}
\begin{proof}
We first notice that notice that $\varepsilon$ defined as above rises the degree by one in $\mathcal{K}_\bullet, \emph{i.e.}$\ $\varepsilon : \mathcal{K}_{i} \rightarrow \mathcal{K}_{i+1}$ since it multiplies by the elements $\chi_i$ and $\ell_j$. Also, it is \emph{odd} - likewise $\delta$ - and it is symmetric to $\delta$, in that it is a derivation on the first factor of $R \otimes_A R^{\pi} = \mathcal{K}_\bullet$, whilst $\delta$ was a derivation on the second factor of $R \otimes_A {R}^\pi = \mathcal{K}_{\bullet}.$\\
Let us now compute the commutator $[\delta, \varepsilon] = \delta \circ \varepsilon + \varepsilon \circ \delta$. By applying the definitions, one has that
\begin{align} 
& \delta \circ \varepsilon = \sum_{i=1}^p x_i \partial_{x_i} \otimes id_{R^{\pi}} + \sum_{j = 1}^q\theta_j \partial_{\theta_j} \otimes id_{R^\pi},\\
& \varepsilon \circ \delta = id_{R} \otimes \sum_{i = 1}^p \chi_i \partial_{\chi_i} + id_{R} \otimes \sum_{j = 1}^q \ell_j \partial_{\ell_j},
 \end{align}
so that one has 
\bear
[\delta, \varepsilon] = \left ( \sum_{i=1}^p x_i \partial_{x_i} + \sum_{j = 1}^q\theta_j \partial_{\theta_j} \right ) \otimes id_{R^\pi} + id_{R} \otimes \left (\sum_{i = 1}^p \chi_i \partial_{\chi_i} + \sum_{j = 1}^q \ell_j \partial_{\ell_j} \right ),
\eear
which in turn can be rewritten as 
\bear
[\delta, \varepsilon] = E \otimes id_{R^\pi} + id_R \otimes E^\pi,
\eear
where $E$ and $E^\pi$ are the Euler vector fields acting on $R$ and $R^\pi$ respectively. 
\end{proof}
\noindent Upon a suitable normalization, the operator $\varepsilon $ defined in the previous lemma allows us to write the homotopy for the Koszul complex, as we show in the following.
\begin{lemma}[Homotopy of $\delta$] \label{lemma2} Let $(\mathcal{K}_{\bullet}, \delta)$ be the super Koszul complex associated to some $V = A^{p|q}$ with basis as above and let $\varepsilon : \mathcal{K}_{\bullet} \rightarrow \mathcal{K}_{\bullet} $ be defined as in Lemma \ref{lemma1}. Then the operator $\varepsilon_{k,i} \defeq \frac{1}{k+i} \varepsilon $ with
\bear
\xymatrix{
\varepsilon_{k,i} \defeq \frac{1}{k+i} \varepsilon : Sym^k V \otimes Sym^i \Pi V \ar[r] & Sym^{k-1} V \otimes Sym^{k+1} V,
}
\eear
is such that 
\bear
\delta \circ \varepsilon_{k, i} + \varepsilon_{k+1, i-1} \circ \delta =  {id_R \otimes id_{R^\pi}}. 
\eear
In particular, $h^{\mathcal{K}}_\bullet \defeq \oplus_{k,i \geq 0} \varepsilon_{k,i} : \mathcal{K}_{\bullet} \rightarrow \mathcal{K}_{\bullet}$ with 
\bear
h_i^\mathcal{K} \defeq \oplus_{k \geq 0} \varepsilon_{k,i} : \mathcal{K}_{-i } = R \otimes_A R_i^\pi  \longrightarrow \mathcal{K}_{-i-1} = R \otimes_A R^{\pi}_{i+1}
\eear
defines a homotopy for the differential of the super Koszul complex. 
\end{lemma}
\begin{proof}
Notice that $\varepsilon_{k,i} : R_k \otimes R_{i}^\pi \rightarrow R_{k-1} \otimes R_{i+1}^\pi$ is just a normalization of $\varepsilon.$ Further, notice that, symmetrically, the Koszul differential $\delta$ acts on the homogeneous factors of $\mathcal{K}_{\bullet} = R \otimes_A R^\pi$ as $\delta : R_k \otimes R_{i}^\pi \rightarrow R_{k+1} \otimes R^\pi_{i-1}.$ 
We can thus consider the following diagram
\bear
\xymatrix{
\cdots \ar[r] & R_{k-1} \otimes R^\pi_{i+1} \ar[r] & R_{k} \otimes R_{i} \ar@<-1ex>[d] \ar@<1ex>[d] \ar[dl]_{\varepsilon_{k,i}} \ar[r]^{\delta_i\; \; \;} & R_{k+1} \otimes R_{i-1}\ar[dl]^{\varepsilon_{k+1, i-1 }} \ar[r] & \cdots \\
\cdots \ar[r] & R_{k-1} \otimes R^\pi_{i+1}  \ar[r]^{\; \; \; \delta_{i-1}} & R_{k} \otimes R_{i} \ar[r] & R_{k+1} \otimes R_{i-1} \ar[r] & \cdots
}
\eear 
We first observe that, as for the normalization, one has 
\bear
\varepsilon_{k,i} = 
\frac{1}{k+i} \varepsilon = \varepsilon_{k+1,i-1}.
\eear
Therefore, it follows from the computation of the commutator in the previous lemma \ref{lemma1} that we can set
\begin{align}
\delta \circ \varepsilon_{k, i} + \varepsilon_{k+1, i-1} \circ \delta = \frac{1}{k+i} [\delta, \varepsilon ] = \frac{1}{k+i} \left (E \otimes id_{R^\pi} + id_R \otimes E^\pi \right ).
\end{align}
Applying lemma \ref{Euler}, since the degrees corresponds to the powers $k$ and $i$ for $R_k$ and $R^{\pi}_i$ respectively, one gets,
\begin{align}
\frac{1}{k+i} E \otimes id_{R^\pi} + id_R \otimes E^\pi = \frac{k+i}{k+i} {id_R \otimes id_{R^\pi}} = {id_R \otimes id_{R^\pi}}, 
\end{align}
establishing the first part of the lemma. For the second part it is enough to observe that the $(-1)$-th terms of the super Koszul complex can be written as $
\mathcal{K}_{-i} = R \otimes_A Sym^i \Pi V = \left ( \bigoplus_{k\geq 0} R_k \right ) \otimes_A  R^\pi_i
$. 
\end{proof}
\noindent The previous lemma allows to compute the homology of the super Koszul complex, the main result of the present section.
\begin{theorem}[Homology of $\mathcal{K}_\bullet$] \label{homologykos} The homology of the super Koszul complex $(\mathcal{K}_\bullet, \delta)$ is given by
\bear
H_i (( \mathcal{K}_\bullet, \delta ) ) \cong \left \{ \begin{array}{ccc}
A   & &  i = 0\\
0  & & i \neq 0.
\end{array}
\right.
\eear
In particular $(\mathcal{K}_\bullet, \delta)$ is an exact resolution of $A$ endowed with the structure of $R$-module.
\end{theorem}
\begin{proof} It is enough to observe that the homotopy operator $h_\bullet^\mathcal{K} : \mathcal{K}_\bullet \rightarrow \mathcal{K}_\bullet$ introduced in the previous lemma \ref{lemma2} is defined for $k+i > 0.$ It follows that, when $h^\mathcal{K}_\bullet$ is defined, \emph{i.e.}\ for any $k+i \geq 0$, we have 
\bear
H_i (\mathcal{K}_\bullet)_k = 0,
\eear
where $H_i (\mathcal{K}_\bullet)_k$ is the $R$-degree $k$-component of $H_i (\mathcal{K}_\bullet)$, with the structure of $\mathbb{Z}$-graded $R$-module inherited by that of $\mathcal{K}_\bullet$. The only homology group of $\mathcal{K}_{\bullet} $ which is not annihilated by the homotopy corresponds to the choice $i=0=k$, \emph{i.e.}\ ${H}_0 (\mathcal{K}_\bullet)_0$. In this case the complex reads
\bear
\xymatrix{
\cdots \ar[r] & 0 \ar[r] & 0 \ar[r] & A \ar[r] &  0,
}
\eear
so that $H_0 (\mathcal{K}_\bullet)_0 = A,$ and the result follows.
\end{proof}
\begin{remark} \label{remark} Before we pass to the next section, a remark is now in order. Indeed the ring $R$ entering in the construction of the super Koszul complex is a ring of polynomials in variables that are by no means a \emph{regular sequence} in general, due to the presence of odd variables. This is the reason why the usual induction proof - that can be found for example in \cite{Eisenbud} -  of the analog of Theorem \ref{homologykos} in the commutative case does \emph{not} extend to the supercommutative case, thus leading us to make use of the homotopy previously constructed in Lemmas \ref{lemma1} and \ref{lemma2}. It is worth noticing by the way that the above computation of the homology breakdown in the case the characteristic of $A$ is different than zero, leading to an interesting and richer scenario as the following example shows.
\end{remark}
\begin{example}[Homology of Super Koszul Complex in $\mbox{char}(A ) = p$] Let us set $\mbox{char} (A) = 3$, for example $A \defeq \mathbb{Z}_3$ and let us consider two variable, $x$ even and $\theta$ odd and set, as above, $\chi = \pi x$ and $\ell = \pi \theta$. We therefore have $V = \mathbb{Z}_3 x \oplus \mathbb{Z}_3 \theta = \mathbb{Z}_3^{1|1}$ and hence 
\bear
R \defeq Sym^\bullet V = \mathbb{Z}_3 [x | \theta], \qquad R^\pi \defeq Sym^\bullet \Pi V = \mathbb{Z}_3 [\ell | \chi]. 
\eear
This leads to consider the super Koszul complex given by $R\otimes_{\mathbb{Z}} R^\pi = \mathbb{Z}_3 [x, \ell | \chi, \theta]$, having differential defined by $\delta = x \partial_{\chi} + \theta \partial_{\ell}.$ In the notation $(\mathcal{K}_\bullet , \delta)$ we have that 
\bear
\xymatrix{
\ldots \ar[r] & \mathcal{K}_{-3} \ar[r]^{\delta} & \mathcal{K}_{-2} \ar[r]^{\delta} &  \mathcal{K}_{-1} \ar[r]^{\delta} & \mathcal{K}_0 \ar[r] & 0,
}
\eear
corresponds to 
\bear
\xymatrix{
\ldots \ar[r] & R \cdot \ell^3 \oplus R \cdot \ell^2 \chi \ar[r]^{\delta} & R\cdot \ell^2 \oplus R\cdot \ell \chi \ar[r]^{\delta} &  R \cdot \ell \oplus R \cdot \chi \ar[r]^{\qquad \delta} & R \ar[r] & 0.
}
\eear
Now consider the element $\theta \ell^2 \in \mathcal{K}_{-2}$: it is straightforward to check that it is a cycle, \emph{i.e.} $\delta (\theta \ell^2) = 0$. On the other hand, for any element $\tau \in \mathcal{K}_{-3}$, with $
\tau \defeq f (x | \theta ) \, \ell^3 + g (x | \theta) \, \ell^2 \chi \in \mathcal{K}_{-3},$
one has that $\delta (\tau) =  2 \theta g (x| \theta) \ell \chi + x g(x | \theta) \ell^2 $, so that in particular $\theta \ell^2$ is not a boundary, \emph{i.e.} there is no $\tau \in \mathcal{K}_{-3}$ such that $\theta \ell^2 \neq \delta (\tau)$ and therefore $[\theta \ell^3 ] \in H_2 ((\mathcal{K}_\bullet, \delta )) \neq 0.$
\end{example}

\section{The Dual of the Super Koszul Complex and its Homology}

\noindent Given the super Koszul complex $(\mathcal{K}_\bullet, \delta)$ associated to $V= A^{p|q}$, we can define its \emph{dual} via the functor $Hom_R (- , R)$, for $R = Sym^\bullet V$ as above. Doing so, one gets the pair
$
(\mathcal{K}_\bullet^\ast, \delta^\ast ) \defeq (Hom_R (\mathcal{K}_\bullet, R), Hom_R (\delta, R)). 
$
Defining 
\bear
R^{\pi \ast}_k \defeq \bigoplus_{k \geq 0 } R^{\pi \ast}_i \quad \mbox{with} \quad R^{\pi \ast}_k \defeq Sym^k \Pi V^\ast
\eear
the complex $\mathcal{K}^\ast_\bullet$ is thus given by
\bear
\mathcal{K}_{\bullet}^\ast = \bigoplus_{k \geq 0} \mathcal{K}_{k}^\ast = R \otimes_A \bigoplus_{k \geq 0} R_k^{\pi \ast} = R \otimes_A R^{\pi \ast} 
\eear
Notice that $R^{\pi \ast}$ is a $A$-superalgebra generated by the elements $\{ \partial_{\ell_1}, \ldots, \partial_{\ell_q} | \partial_{\chi_1}, \ldots, \partial_{\chi_p} \}$ for $\ell_i \defeq \pi \theta_i$ and $\chi_j \defeq \pi x_j$ for $i = 1, \ldots, q$ and $j = 1, \ldots, p.$ \\
The operator $\delta^\ast : \mathcal{K}^\ast_{\bullet} \rightarrow \mathcal{K}^\ast_{\bullet}$ is formally identical to $\delta : \mathcal{K}_{\bullet} \rightarrow \mathcal{K}_{\bullet}$, the differential of the super Koszul complex introduced above, but what is crucial to note is that $\delta^\ast$ should now be seen as the \emph{multiplication operator} by the \emph{odd} element $\sum_j x_j \otimes \partial_{\chi_j } + \sum_i \theta_i \otimes \partial_{\ell_i}$ in the superalgebra $R \otimes_A R^{\pi \ast}$. Once this is acknowledged, we still write $\delta^\ast$ as
\bear
\delta^\ast \defeq \sum_{j=1}^p x_j \otimes \partial_{\chi_j } + \sum_{i=1}^q \theta_i \otimes \partial_{\ell_i}.
\eear  
Further, note that since $\delta^\ast$ acts as the multiplication by an odd element, it is automatically \emph{nilpotent}, \emph{i.e.}\ $\delta^\ast \circ \delta^\ast = 0$: this justifies the following definition.
\begin{definition}[Dual of the Super Koszul Complex] Given any free $A$-module $V= A^{p|q}$ for any superalgebra $A$, we call the pair $(\mathcal{K}_{\bullet}^\ast, \delta^\ast)$ the dual of the super Koszul complex associated to $V$.
\bear
\xymatrix{ 
0 \ar[r] & R \ar[r]^{\delta^\ast \quad} & R \otimes \Pi V^\ast \ar[r]^{\delta^\ast \quad} & R \otimes Sym^2 \Pi V^\ast \ar[r]^{\qquad \quad \delta^\ast} & \ldots \ar[r]^{\delta^\ast\qquad \; \;} & R \otimes Sym^k \Pi V^\ast \ar[r]^{\quad \qquad \delta^\ast} & \ldots 
}
\eear
\end{definition}
\noindent Before we go on and study the homology of this complex, let us briefly discuss the \emph{functoriality} of the above construction, as to understand the properties of the functor $V \mapsto \mathcal{K}^{\ast}_{\bullet}.$ \\
Given two $A$-supermodules $V$ and $W$, applying the functor $Sym^\bullet (-) : \mathbf{SMod}_A \rightarrow \mathbf{SAlg}_A $, one gets the $A$-superalgebras $R^V \defeq \bigoplus_{i\geq 0} Sym^i V,$ and $R^W \defeq \bigoplus_{i\geq 0} Sym^i W$. As for the arrows, given a homomorphism $f: V \rightarrow W$ of $A$-supermodules, the action of the functor yields a supercommutative $A$-algebra morphism:
\bear
f \longmapsto Sym^\bullet (f) : R^V \longrightarrow R^W.
\eear
Likewise, considering a second homomorphism $f^{\pi} : \Pi V \rightarrow \Pi W$, upon applying $Sym^\bullet (-)$ to the direct sum $f \oplus f^{\pi} : V \oplus \Pi V \rightarrow  W \oplus \Pi W$, one gets
\bear
f \oplus f^{\pi} \longmapsto Sym^{\bullet} (f \oplus f^{\pi}) : R^V \otimes_A R^{V \pi} \longrightarrow R^W \otimes_A R^{W\pi},
\eear
which corresponds to a morphism between the super Koszul complex associated to $V$ - we call it $\mathcal{K}_\bullet^V$ - and the super Koszul complex associated to $W$ - we call it $\mathcal{K}^{W}_\bullet$:
\bear
f^{VW}_\bullet \defeq Sym^{\bullet} (f \oplus f^{\pi}) : \mathcal{K}^V_{\bullet} \longrightarrow \mathcal{K}_{\bullet}^W.
\eear
Let us apply the functor $Hom_{R^W} (- , R^W)$: one has the following commutative triangle
\bear
\xymatrix{
\mathcal{K}^V_\bullet \ar[dr]_{h^{W}( {\mathcal{K}_\bullet^V)}} \ar[rr]^{f^{VW}_\bullet} && \mathcal{K}^W_\bullet \ar[dl]^{h^{W}( {\mathcal{K}_\bullet^W)}}  \\
& R^W.
}
\eear
The action on the functor on the morphisms gives the following map
\bear \label{functor}
\xymatrix@R=1.5pt{
(f^{VW}_\bullet)^\ast : Hom_{R^W} (\mathcal{K}^W_\bullet, R^W) \ar[r] & Hom_{R^W} (\mathcal{K}_\bullet^V, R^W)\\
h^{W}( {\mathcal{K}_\bullet^W)} \ar@{|->}[r] & h^{W} (\mathcal{K}_\bullet^V) \defeq h^{W}( {\mathcal{K}_\bullet^W)} \circ f^{VW}_{\bullet},
}
\eear
where we have defined $(f^{VW}_\bullet)^\ast \defeq Hom_{R^W} (f^{VW}_{\bullet} , R^W).$ It follows that the functor $V \mapsto \mathcal{K}_{\bullet}^\ast$ is \emph{not} strictly a contravariant functor. Indeed observing that 
\bear
Hom_{R^W} (\mathcal{K}_\bullet^V, R^W) = Hom_{R^V} (\mathcal{K}_\bullet^V, R^V) \otimes_{R^V} \otimes R^W = \mathcal{K}^{V^\ast}_\bullet \otimes_{R^V} R^W,
\eear the previous \eqref{functor} reads
\bear
\xymatrix@R=1.5pt{
(f^{VW}_\bullet)^\ast : \mathcal{K}^{W\ast}_\bullet \ar[r] & \mathcal{K}_\bullet^{V\ast} \otimes_{R^V} R^W.
}
\eear
Nonetheless, let us assume that $f \in Aut (V)$ - for example, $f$ is a change of basis. Then, in this case, we have a map
$
(f^{V}_\bullet)^\ast : \mathcal{K}^{V\ast}_{\bullet} \longrightarrow \mathcal{K}^{V\ast}_\bullet. 
$ More precisely one finds
\bear \label{contr}
(f^{V}_\bullet)^\ast = Sym^\bullet (f \oplus (f^\pi)^t) : \mathcal{K}^{V\ast}_\bullet \longrightarrow \mathcal{K}^{V\ast}_{\bullet}.
\eear
and one gets a contravariant functor $V \mapsto \mathcal{K}^{V\ast}_{\bullet}$ with $(f^{V}_\bullet \circ g^{ V}_\bullet )^\ast = (g^{V}_\bullet )^{\ast} \circ (f^{ V}_\bullet)^\ast, $ as can be readily checked. Finally, notice that because of their particular form, the homological operators $\delta$ and $\delta^\ast$, differential of the super Kozsul complex and its dual respectively, are \emph{invariant} under change of basis, \emph{i.e.}\ automorphisms of $V$. \\

\noindent As done in the previous section for the super Koszul complex $(\mathcal{K}_{\bullet}, \delta)$, we are now interested into computing the homology of the dual of the super Koszul complex $(\mathcal{K}_\bullet^\ast, \delta^\ast)$. \\
To this end, recalling that the operator $\delta^\ast = \sum_i x_i \otimes \partial_{\chi_i} + \sum_j \theta_j \otimes \partial_{\ell_j}$ is now looked as a \emph{multiplication} operator in the algebra $\mathcal{K}_{\bullet}^\ast= R \otimes R^{\pi \ast}$,
one immediately gets an \emph{inclusion of ideals}. \\
Indeed, clearly, $(\delta^\ast)^2  = 0$, so the element corresponding to 
\bear
\delta^\ast \defeq \sum_{j=1}^p x_j \otimes \partial_{\chi_j } + \sum_{i=1}^q \theta_i \otimes \partial_{\ell_i} \in V \otimes \Pi V^\ast 
\eear 
is in the kernel of $\delta^\ast$ seen as the multiplication operator, \emph{i.e.} $\delta^\ast \subseteq \ker (\delta^\ast)$. \\
On the other hand it is immediate to observe that also the element 
\bear 
\mathcal{D} \defeq \theta_1\ldots \theta_q \otimes \partial_{\chi_1} \ldots \partial_{\chi_p} \in Sym^q V \otimes Sym^p \Pi V^\ast
\eear 
is in the kernel of $\delta^\ast$, since both the factor in $Sym^q V$ and the factor in $Sym^p \Pi V^\ast$ are completely antisymmetric and when they get multiplied by another odd term coming from $\delta$ they yield zero, \emph{i.e.}\ $\mathcal{D} \in \ker (\delta^\ast).$ We therefore have the following inclusion of ideals 
\bear
\left (\delta^\ast, \;  \mathcal{D} \right ) \subseteq \mbox{Ker} (\delta^\ast),
\eear
where, as it is customary, $(\delta^\ast, \mathcal{D})$ denotes the ideal generated by the elements $\delta^\ast$ and $\mathcal{D}$. We now prove that such an inclusion is indeed an equality. 
\begin{lemma} \label{kerteo} Let $(\mathcal{K}^\ast, \delta^\ast)$ the dual of the super Koszul complex associated to $V = A^{p|q}$ for some $A$. Then $
\ker (\delta^\ast) = \left ( \delta^\ast, \mathcal{D} \right )$. In particular 
\bear
\ker (\delta^\ast) / \mbox{\emph{im}}\, (\delta^\ast) = (\mathcal{D}).
\eear
\end{lemma}
\begin{proof} The proof of this lemma relies on the ordinary Koszul complex construction of the first section. Let us start simplifying the notation. It can be observed that, by definition
\bear
R \otimes_A R^{\pi \ast} = A [x_1, \ldots x_p, \partial_{\ell_1}, \ldots, \partial_{\ell_q} | \partial_{\chi_1}, \ldots, \partial_{\chi_p} , \theta_1, \ldots, \theta_q ],
\eear
where the even and odd generators have been grouped together. Posing $N \defeq p+q$ we define
\begin{align}
& (u_1, \ldots, u_N) \defeq \left ( x_1, \ldots x_p, \partial_{\ell_1}, \ldots, \partial_{\ell_q} \right ), \nonumber \\
& (\psi_1, \ldots, \psi_N) \defeq \left (  \partial_{\chi_1}, \ldots, \partial_{\chi_p} , \theta_1, \ldots, \theta_q \right ),
\end{align}
so that, upon this redefinition, $\delta^\ast$ and $\mathcal{D}$ read  
\begin{align}
& \delta^\ast \defeq \sum_{i=1}^N u_i \psi_i \in A[u_1, \ldots, u_N , \psi_1, \ldots, \psi_N ], \nonumber \\
& \mathcal{D} = \prod_{j=1}^N \psi_i \in A[u_1, \ldots, u_N , \psi_1, \ldots, \psi_N].
\end{align}
as elements of the ring $A[u_1, \ldots, u_N , \psi_1, \ldots, \psi_N ]$.
Let us pose $B \defeq A [u_1, \ldots, u_N]$ so that one has $A[u_1, \ldots, u_N, \psi_1, \ldots, \psi_N] = B[\psi_1, \ldots, \psi_N]$. By anticommutativity of the $\psi_i$'s one has
\bear
B[\psi_1, \ldots, \psi_N] = B \oplus \sum_{i=1}^N B \cdot \psi_i \oplus \ldots \oplus B \cdot (\psi_1 \ldots \psi_N),
\eear
but this is nothing but the ordinary (dual of the) Koszul complex $K_\bullet$ introduced above in \eqref{Rkos} and the result follows from theorem \ref{kosclas}, that proved that the homology is generated over $A$ by the element $\mathcal{D} =  \psi_1 \ldots \psi_N.$ 
\end{proof}

\noindent Recalling that we have proved in the previous section that the super Koszul complex $\mathcal{K}_\bullet$ is an exact resolution of $A$ seen as $R$-module, 
and that $Ext^{i}_R (A, R) \defeq H^i (Hom_R (\mathcal{K}_\bullet, R)) = H^i ((\mathcal{K}_\bullet^\ast, \delta^\ast ))$, we can finally compute the homology of the dual of super Koszul complex.
\begin{theorem}[Homology of $\mathcal{K}^\ast_\bullet$] \label{dualhomology}The homology of the dual of the Koszul supercomplex $(\mathcal{K}_\bullet^\ast, \delta^\ast)$ is given by
\bear
Ext^i_R (A, R) = \left \{ \begin{array}{ccl}
\Pi^{p+q} A  & & i = p\\
0  & & \mbox{else}.
\end{array}
\right.
\eear
In particular, in the above notation, a generator is given by the element $\theta_1 \ldots \theta_q \otimes \partial_{\chi_1 } \ldots \partial_{\chi_p} \in \mathcal{K}_p^\ast.$
\end{theorem}
\begin{proof} Recovering the original notation, one has the following correspondence
\bear
\psi_1 \ldots \psi_N = \theta_1 \ldots \theta_q \otimes \partial_{\chi_1} \ldots \partial_{\chi_p} \in \mathcal{K}^\ast_\bullet = R \otimes_A R^{\pi\ast}
\eear
More precisely one sees that $\theta_1 \ldots \theta_q \otimes \partial_{\chi_1} \ldots \partial_{\chi_p} \in Sym^q V \otimes_A Sym^p \Pi V^\ast $, which implies that $ \theta_1 \ldots \theta_q \otimes \partial_{\chi_1} \ldots \partial_{\chi_p} \in\mathcal{K}^\ast_p$. By lemma \ref{kerteo} above, it generates the cohomology of $\mathcal{K}^\ast_\bullet $ as a $A$-supermodule and its parity depends on the sum $N=p+q$, so that the conclusion follows.
\end{proof}

\section{Super Koszul Complex and the Berezinian }

\noindent We now aim at interpreting the main result of the previous section, and we show that $Ext^p_R (A, R)$ transforms exactly as the (inverse of the) \emph{Berezinian module} of $V$. More precisely, we prove the following theorem.
\begin{theorem} Let $\phi \in Aut_A (V)$ be an automorphism of the free $A$-supermodule $V$. Then the induced automorphism $\widehat \varphi \in Aut_A (Ext^p_R (A, R))$ is given by the multiplication by the \emph{inverse} of the Berezinian of the automorphism $Ber (\varphi)^{-1}$, 
\bear
\xymatrix@R=1.5pt{
\widehat{\varphi} : Ext^p_R (A, R) \ar[r] & Ext^p_R (A, R) \\
\mathcal{D} \ar@{|->}[r] & Ber (\varphi)^{-1} \cdot \mathcal{D}
}
\eear
\end{theorem}
\begin{proof} Let us fix the base $\{x_1, \ldots, x_p | \theta_1, \ldots, \theta_q \}$ for $V$. Then, $\varphi \in Aut_A (V)$ is represented by an invertible matrix $[M] \in GL(p|q, A)$
\bear
[M (\varphi)]_{\alpha \beta} = \left ( 
\begin{array}{c|c}
A & B \\
\hline 
C & D
\end{array}
\right ) = \left ( 
\begin{array}{c|c}
a_{hi} & b_{hj} \\
\hline c_{ki} & d_{kj}
\end{array}
\right )
\eear
where $A \in GL(p, A_0), D\in GL (q, A_0)$ are even and $B \in Hom (A^p, A^q)$, $C \in Hom ( A^q, A^p)$ are odd submatrices, such that one has the following transformations
\begin{align}
x^\prime_i \defeq \varphi (x_i) = \sum_{h=1}^p x_h a_{hi} + \sum_{k=1}^q \theta_k c_{ki}, \\
\theta^{\prime}_j \defeq \varphi (\theta_j) = \sum_{h=1}^p x_h b_{hj} + \sum_{k= 1}^q \theta_k d_{kj}.
\end{align}
We recall that if $\varphi_i$ for $i=1,2$ are automorphisms of $V$, we have a contravariant functorial construction (see the remarks around \eqref{contr}), such that $(\varphi_1 \circ \varphi_2)^\ast = \phi_2^\ast \circ \varphi_1^\ast : \mathcal{K}_{\bullet}^\ast \rightarrow \mathcal{K}_{\bullet}^\ast,$ therefore the product of two matrices $M (\varphi_1) \cdot M (\varphi_2) \in GL (p|q, A)$ corresponds to the product of two elements $ \widehat{\varphi}_{M(\varphi_2)} \cdot \widehat{\varphi}_{M(\varphi_1)} $ acting as automorphisms of $Ext^p_R (A,R) = H^p (\mathcal{K}^\ast_\bullet).$
We can thus use the decomposition 
\bear
\left (
\begin{array}{c|c}
A & B \\
\hline 
C & D 
\end{array} 
\right )= \left ( 
\begin{array}{c|c}
1 & BD^{-1} \\
\hline 
0 & 1
\end{array}
\right )
\left ( 
\begin{array}{c|c}
A - BD^{-1}C & 0 \\
\hline 
0 & D
\end{array}
\right )
\left ( 
\begin{array}{c|c}
1 & 0 \\
\hline 
D^{-1}C & 1
\end{array}
\right )
\eear
to reduce ourselves to the following cases.
\bear
(1): \;  M = \left (
\begin{array}{c|c}
A & 0 \\
\hline 
0 & D 
\end{array} 
\right ), \qquad
(2): \;  M = \left (
\begin{array}{c|c}
1 & 0 \\
\hline 
\ast & 1 
\end{array} 
\right ), \qquad
(3):  \;  M = \left (
\begin{array}{c|c}
1 & * \\
\hline 
0 & 1 
\end{array} 
\right ).
\eear 
Let us consider the class of the generator of the homology, we call it $ \mathcal{D} = \theta_1 \ldots \theta_q \otimes \partial_{\chi_1} \ldots \partial_{\chi_p} \in Ext^p_R (A, R)$ as above, and let us examine its transformation under automorphisms of the form (1), (2), (3) separately. \vspace{.2cm}\\
\noindent (1)  In this case, it is simply to see that $\mathcal{D}$ transforms as $\det(D) \cdot \det (A)^{-1}$, as one has factorization of the transformations of the $\theta$'s - contributing with $\det (D)$ and of the $\partial_{\chi}$'s - contributing with $\det(A)^{-1}$.\vspace{.2cm}\\
(2) In this case one can observe that a generic automorphism of this forms is a composition of elementary automorphisms of the forms
\begin{align}
& x^\prime_i = x_i + \alpha \theta_k, \qquad  \quad \alpha \in A_1  \\
& x^\prime_l = x_l,   \qquad \qquad \qquad l \neq i, \\ 
& \theta^\prime_j = \theta_j.
\end{align}
It follows that one finds the following transformations: 
\begin{align}
& \partial_{x_l} = \partial_{x^\prime_l},  \qquad \qquad \forall l \\
& \partial_{\theta_j} = \partial_{\theta^\prime_j}, \qquad \qquad j \neq k \\
& \partial_{\theta_k} = \partial_{\theta_k^\prime} + \alpha \partial_{x^\prime_i}.
\end{align}
Thus, recalling that $\partial_{\chi_i} = \partial_{\pi x_i}$ and that $\partial_{\ell_j} = \partial_{\pi \theta_j}$, these are rewritten as
\begin{align}
& \partial_{\chi_l} = \partial_{\chi^\prime_l},  \qquad \qquad \forall l \\
& \partial_{\ell_j} = \partial_{\ell^\prime_j}, \qquad \qquad j  \neq k \\
& \partial_{\ell_k} = \partial_{\ell_k^\prime} + \alpha \partial_{\chi^\prime_i}.
\end{align}
In particular, one sees that $\mathcal{D}$ is invariant under these transformation.\vspace{.2cm}\\
(3) This is similar to the previous case. Indeed it can be observed again that a generic automorphism of this form is a composition of elementary automorphisms of the forms
\begin{align}
& x^\prime_i = x_i  \\
& \theta^\prime_j = \theta_j +  \beta x_k, \quad \qquad \beta \in A_1, \\ 
& \theta^\prime_m = \theta_m, \qquad \qquad \quad \, m \neq j.
\end{align}
Thus, similarly, one finds that
\begin{align}
& \partial_{\chi_k} = \beta \partial_{\ell^\prime_j} + \partial_{\chi_k^\prime},  \\
& \partial_{\chi_i} = \partial_{\chi^\prime_i}, \qquad \qquad i  \neq k \\
& \partial_{\ell_j} = \partial_{\ell_j^\prime} \qquad \quad \quad \; \; \forall j.
\end{align}
It follows that the transformations reads
\begin{align}
\theta_1\ldots \theta_q \otimes \partial_{\chi_1} \ldots \partial_{\chi_p} = & \; \theta^\prime_1\ldots \theta^\prime_q \otimes \partial_{\chi^\prime_1} \ldots \partial_{\chi^\prime_p} + \nonumber \\
& + (-1)^{j+1} \beta x_k \theta_1 \ldots \hat \theta_j \ldots \theta_q \otimes \partial_{\chi^\prime_1} \ldots \partial_{\chi^\prime_p} + \nonumber \\
& + (-1)^{k+1} \theta_1^\prime \ldots \theta_q^\prime \otimes \partial_{\ell_j} \partial_{\chi_1} \ldots \hat \partial_{\chi_k} \ldots \partial_{\chi_p},
\end{align}
where the element marked with a hat is missing. Noticing that the second and the third elements are of lower degree either in the $\theta$'s or in the $\partial_\chi$'s with respect to the generator $\mathcal{D} \in Ext^p_R (A, R)$, hence they do not contribute to the transformation of the homology class, which is again invariant.  \\
Combining the above cases, it follows that one has that $\widehat \varphi_{M(\varphi)} = \det (D) \det (A- BD^{-1}C)^{-1}$, which is nothing but $Ber(M (\varphi))^{-1}$ as claimed.
\end{proof}
\noindent The previous theorem shows that an automorphism $\varphi_V \in Aut_A (V)$ induces an automorphism $\hat \varphi \in Aut_A (Ext^p_R (A, R))$ that is given by the multiplication of the inverse of the Berezinian of the trasformation,
\bear
\xymatrix@R=1.5pt{
\widehat \varphi_V : Ext^p_{Sym^\bullet V} (A, Sym^\bullet V) \ar[r] & Ext^p_{Sym^\bullet V} (A, Sym^\bullet V) \\
\mathcal{D} \ar@{|->}[r] & Ber (\varphi_V)^{-1} \cdot \mathcal{D}. 
}
\eear 
Clearly, if one considers instead of $V$ its \emph{dual} $V^\ast = Hom_A (V, A)$ and the related homology of the dual of the super Koszul complex one finds
\bear
\xymatrix@R=1.5pt{
\widehat \varphi_{V^\ast} : Ext^p_{Sym^\bullet V^\ast} (A, Sym^\bullet V^\ast) \ar[r] & Ext^p_{Sym^\bullet V^\ast} (A, Sym^\bullet V^\ast) \\
\mathcal{D}^\ast \ar@{|->}[r] & Ber (\varphi_V) \cdot \mathcal{D}^\ast. 
}
\eear 
\noindent This remark naturally lead to the following definition.
\begin{definition}[Berezinian of a Free $A$-Module] Let $V$ be a free $A$-supermodule of rank $p|q$ for $A$ any superalgebra. Then we call the Berezinian of $V$ the free $A$-supermodule of rank $\delta_{0, (p+q)\mbox{\scriptsize{mod}} 2} | \delta_{1, (p+q)\mbox{\scriptsize{mod}}2} $ given by
\bear
Ber (V) \defeq Ext^p_R (A, R) \cong \Pi^{p+q}A, 
\eear
where $R \defeq Sym_{A}^\bullet V^\ast$ and $A = R / I_{max}$ for $I_{max} \defeq \bigoplus_{k\geq 0} Sym^k V^\ast$.
\end{definition} 
\noindent Then, by the previous remark, an automorphism $\varphi :V \rightarrow V$ induces an automorphism $Ber (\varphi) : Ber (V) \rightarrow Ber (V)$ which is given by the multiplication by $Ber (\phi)$ and such that $Ber (\varphi_1 \circ \varphi_2) = Ber (\varphi_2) Ber (\varphi_1).$ \\

\noindent As a conclusive remark, let us stress that this construction \virgolette parallels'' in superalgebra the ordinary construction of the determinant or canonical module of a free module $V$ via homology of its related Koszul complex. In this sense, it should be clear the deep meaning behind the \virgolette slogan'' that \emph{the Berezinian replaces the determinant} when passing from a commutative setting to a supercommutative setting: notice further that the provided construction of the Berezinian module via super Koszul complex reduces to the construction of the ordinary determinant module via Koszul complex if $V$ purely even, \emph{i.e.} it is of rank $p|0$. \\
Furthermore, notice that the construction is readily generalizable from algebra to geometry, just by substituting the $Ext$-module with the $\mathcal{E}xt$-sheaf. In particular, over a supermanifold $\mani \defeq (|\mani|, \mathcal{O}_\mani)$ (see \cite{BR, CCF, Manin}) the structure sheaf $\mathcal{O}_\mani$ plays the role of $A$, while the free supermodule $V$ becomes a locally-free sheaf $\mathcal{E}$ of rank $p|q$ over $\mani$ (see \cite{Univ}). In particular, the choice of $\mathcal{E} \defeq \Omega^1_\mani$, the \emph{cotangent sheaf} of $\mani$, leads to what is usually called the \emph{Berezinian sheaf} of $\mani$, \emph{i.e.}\ $\mathcal{B}er (\mani) \defeq \mathcal{E}xt^p_{Sym^\bullet (\Omega^1_\mani)^\ast} (\stsheaf, Sym^\bullet (\Omega^1_\mani)^\ast)$, locally generated by 
\bear
\mathcal{D} = dx_1 \ldots dx_p \otimes \partial_{\theta_1} \ldots \partial_{\theta_q} \in \mathcal{B}er (\mani) \defeq \mathcal{E}xt^p_{Sym^\bullet (\Omega^1_\mani)^\ast} (\stsheaf, Sym^\bullet (\Omega^1_\mani)^\ast)
\eear
if $\{ x_1, \ldots, x_p | \theta_1, \ldots, \theta_q \}$ are local coordinates for $\mani$. Notice that this holds true in any geometric category, either \emph{smooth, analytic} of \emph{algebraic}.

\end{document}